\documentclass{article}
\usepackage[french,english]{babel}
 \usepackage[T1]{fontenc}        
 \usepackage[latin1]{inputenc}   
 \usepackage{amssymb,amsfonts,amsthm,amsmath}
 \usepackage{graphicx}
 
 \newtheorem{theorem}{Theorem}

 \newtheorem{corollary}[theorem]{Corollary}

 \title{Calculation of norms of some special elements of cyclotomic fields.}
 \author{Alexandre Aksenov}
 \date{}

\newcommand{\Q}{\mathbb Q}

\begin{document}
\maketitle
\begin{abstract}
In this article we prove that $1-\zeta+\zeta^2$ is a unit in the ring of integers of the cyclotomic field where $\zeta$ is a primitive $n$-th root of unity and $n$ is coprime to $2$ and $3$. We also prove that for prime $n$, $N_{\Q(\zeta)/\Q}(1-\zeta-\zeta^2)=L(p)$ the $p$-th Lucas number thus completing the study of norms of quadratic polynomials in $\zeta$ that only have coefficients equal to $1$ or $-1$ and both numbers appear.  
\end{abstract}

\vspace{0.5cm}


\vspace{0.5cm}

\section*{1. Introduction.}
As shown in \cite{Aksenov}, the study of the Newman phenomenon in the multiplicative sequences leads naturally to the study of the norms (over $\Q$ or over a subfield of small degree) of numbers of the form $r(\zeta)$ where $\zeta$ is a primitive $p$-th root of unity ($p$ being prime) and $r$ is a polynomial with coefficients $1$ or $-1$, the free term being $1$; the only relevant polynomials for the study of the Newman phenomenon are those that have at least one $-1$ coefficient. We are looking for results that are valid for a fixed $r$ and an arbitrary prime  $p$ (the conditions that can be put on $p$ are: the prime $p$ allows the existence of the desired subfield and $p$ is big enough).

Here are some results of this kind: for any odd prime $p$ we get $N_{\Q(p)/\Q}(1-\zeta)=p$ (see \cite{NumTh}) and 
\begin{equation}
\label{N++-}
N_{\Q(p)/\Q}(1+\zeta-\zeta^2)=L(p) \text{ the $p$-th Lucas number .}
\end{equation}
(see  \cite{Aksenov}). Moreover, if $p\equiv1$ mod $4$, then $N_{\Q(\zeta)/Q(\sqrt{p})} (1-\zeta)=\sqrt{p}\epsilon^{\pm h}$ where $\epsilon$ is the fundamental unit of the ring of integers of $Q(\sqrt{p})$ and $h$ is its class number; if $p\equiv-1$ mod $4$ we get $N_{\Q(\zeta)/Q(\sqrt{-p})}=\epsilon i\sqrt{p} $ where $\epsilon\in\{1,-1\}$ is a sign defined as follows: suppose $\zeta=e^{\frac{2i\pi k}{p}}$, then $\epsilon=\genfrac{(}{)}{}{}{k}{p}(-1)^{\frac{h+1}{2}}$ where $h$ is the class number of $\Q(\sqrt{-p})$ (see \cite{BorevichShafarevich}).

In this article we are going to study the values of $N_{\Q(p)/\Q}(r(\zeta))$ where $r$ is one of the remaining quadratic polynomials, namely $r_1(\zeta)=1-\zeta+\zeta^2$ or $r_2(\zeta)=1-\zeta-\zeta^2$.

\section*{2. Main results.}
The first result concerns the polynomial $r_1$ and it is the following:
\begin{theorem}
Let $n$ be an integer bigger than $4$ and not multiple of $2$ or $3$, and let $\zeta=e^\frac{2i\pi}n$. Then
$$\prod_{k=1}^{n-1}(1-\zeta^k+\zeta^{2k})=1. $$
Therefore, $1-\zeta+\zeta^2$ is a unit in the ring of integers of $\Q(\zeta)$. 
\end{theorem}
\begin{proof}
For each $k\in[1,n-1]$ we get:
$$1-\zeta^k+\zeta^{2k}=\zeta^k\left(\zeta^k+\zeta^{-k}-1\right)=\zeta^k(2\cos\frac{2\pi k}n-1)=\zeta^k \frac{\cos\frac{3\pi k}n}{\cos\frac{\pi k}n}. $$

The product of terms $\zeta^k$ is one, and it can be checked that the products $\prod\limits_{k=1}^{n-1}\cos\frac{3\pi k}n$ and $\prod\limits_{k=1}^{n-1}\cos\frac{\pi k}n$ only differ by permutation of factors.
\end{proof}

For prime numbers $n$ the method used to prove the formula \ref{N++-} makes the theorem \arabic{theorem} is equivalent to the following
\begin{corollary}
Let $p$ be a prime, $p\geqslant 5$. Then the number of ways of putting an even nonzero number of dominos on the circle of length $p$ is equal to the number of ways of putting an odd number of dominos on that circle.
\end{corollary}

For example for $p=11$ there are:\\
$1$ way of putting $0$ dominos on a circle of length $11$;\\
$11$ ways of putting $1$ domino;\\
$44$ ways of putting $2$ dominos;\\
$77$ ways of putting $3$ dominos;\\
$55$ ways of putting $4$ dominos;\\
$11$ ways of putting $5$ dominos.

For $p=17$ there are:\\
$1$ way of putting $0$ dominos on a circle of length $17$;\\
$17$ ways of putting $1$ domino;\\
$119$  ways of putting $2$ dominos;\\
$442$  ways of putting $3$ dominos;\\
$935$  ways of putting $4$ dominos;\\
$1122$  ways of putting $5$ dominos;\\
$714$ ways of putting $6$ dominos;\\
$204$ ways of putting $7$ dominos;\\
$17$ ways of putting $8$ dominos.

For the polynomial $r_2$ we get the following result:
\begin{theorem}
Let $p$ be an odd prime. Then,
$$ N_{\Q(\zeta)/\Q}(1-\zeta-\zeta^2)=L(p).$$
\end{theorem}
\begin{proof}
We prove this norm to be equal to the norm of $1+\zeta-\zeta^2$. Indeed:
$$\prod_{k=1}^{p-1}(1-\zeta^k-\zeta^{2k})=\prod_{k=1}^{p-1}\left(-\zeta^{2k}(-\zeta^{-2k}+\zeta^{-k}+1)\right)=\prod_{k'=1}^{p-1}(1+\zeta^{k'}-\zeta^{2k'})=L(p). $$ 
\end{proof}

\section*{3. Further questions.}
The results presented here finish the study of norms over $\Q$ relative to quadratic polyninomials (which correspond in terms of Newman's phenomenon to $3$-multiplicative sequences). The case of cubic polynomials seems more challenging.

\end{document}